\documentclass[11pt]{amsart}

\usepackage[english]{babel}
\usepackage[latin1]{inputenc}
\usepackage{amssymb}
\usepackage{amsmath}
\usepackage{amsthm}
\usepackage{enumerate}
\usepackage{amsfonts}
\usepackage{color}
\usepackage{graphicx}
\usepackage{multicol}
\usepackage{multirow}
\usepackage{stmaryrd}
\usepackage{mathtools}

\usepackage{tikz}
\tikzstyle{bull}=[circle,draw=black,fill=black!80]
\usetikzlibrary{matrix}
\usetikzlibrary[arrows,snakes,calc,decorations.pathmorphing]

\usepackage[hidelinks]{hyperref}
\hypersetup{colorlinks=true,citecolor=blue,linkcolor=blue}

\allowdisplaybreaks  					

\theoremstyle{plain}
\newtheorem{theorem}{Theorem}[section]
\newtheorem{proposition}[theorem]{Proposition}
\newtheorem{lemma}[theorem]{Lemma}
\newtheorem{corollary}[theorem]{Corollary}

\theoremstyle{definition}
\newtheorem{definition}[theorem]{Definition}

\newtheorem{remark}[theorem]{Remark}

\numberwithin{equation}{section}

\newcommand{\ph}{\varphi}
\newcommand{\p}{\mathcal{P}}

\newcommand{\Fi}{\mathsf{Fi}}
\newcommand{\Uf}{\mathsf{Uf}}
\newcommand{\Id}{\mathsf{Id}}
\newcommand{\Clop}{\mathsf{Clop}}
\newcommand{\Com}{\mathsf{Com}}
\newcommand{\com}{\mathcal{C}om}
\newcommand{\F}{\mathcal{F}}

\begin{document}

\title[A Characterization of the extension of homomorphisms]{A characterization of the canonical extension of Boolean homomorphisms}

\author{Luciano J. Gonz\'alez}
\address{Universidad Nacional de La Pampa. Facultad de Ciencias Exactas y Naturales. 6300 Santa Rosa. Argentina}

\email{lucianogonzalez@exactas.unlpam.edu.ar}

\thanks{This research was partially supported by Consejo Nacional de Investigaciones Cient\'ificas y T\'ecnicas (Argentina) under the grant PIP 112-20150-100412CO; and it was also partially supported by Universidad Nacional de La Pampa (Facultad de Ciencias Exactas y Naturales) under the grant P.I. 64 M, Res. 432/14 CD}



\begin{abstract} 
This article aims to obtain a characterization of the canonical extension of  Boolean homomorphisms through the Stone-\v{C}ech compactification. Then, we will show that one-to-one homomorphisms and onto homomorphisms extend to one-to-one homomorphisms and onto homomorphisms, respectively. 
\end{abstract}

\maketitle

\section{Introduction}

The concept of canonical extension of Boolean algebras with operators was introduced and studied by J\'onsson and Tarski \cite{JoTa51,JoTa52}. Then, the notion of canonical extension was naturally generalized to the setting of distributive lattices with operators \cite{GeJo94}. Later on, the notion of canonical extension was generalized to the setting of lattices \cite{GeJo00,GeJo04,GeHa01}, and the more general setting of partially ordered sets \cite{DuGePa05}. The theory of completions for ordered algebraic structures, and in particular the theory of canonical extensions, is an important and useful tool to obtain complete relational semantics for several propositional logics such as modal logics, superintuitionistic logics, fragments of substructural logics, etc., see for instance \cite{Bla02,GasaRe00,GeNaVe05,So00,DuGePa05}.

The theory of Stone-\v{C}ech compactifications of discrete spaces and in particular of discrete semigroup spaces has many applications to several branches of mathematics, see the book \cite{HiStr12} and its references. Here we are interested in the characterization of the Stone-\v{C}ech compactification of a discrete space through of the Stone duality for Boolean algebras.

In this paper, we prove a characterization of the canonical extension of a Boolean homomorphism between Boolean algebras \cite{JoTa51} through the Stone-\v{C}ech compactification of discrete spaces and the Stone duality for Boolean algebras. In other words, we present another way to obtain the canonical extension of a Boolean homomorphism using topological tools. Then, we use this characterization to show that one-to-one homomorphisms and onto homomorphisms extent to one-to-one homomorphisms and onto homomorphisms, respectively.

The paper is organized as follows. In Section \ref{sec:duality and extension} we present the categorical dual equivalence between Boolean algebras and Stone spaces, and the principal results about the canonical extension for Boolean algebras. In Section \ref{sec:compactification} we briefly present the theory of Stone-\v{C}ech compactifications. Lastly, Section \ref{sec:characterization} is the main section of the paper, and there we develop the characterization of the canonical extension of Boolean homomorphisms.

\section{Duality theory and canonical extension}\label{sec:duality and extension}

We assume that the reader is familiar with the theory of lattices and Boolean algebras, see for instance \cite{DaPri02,Ko89} . We establish some notations and results that we will need for what follows. Given a lattice $L$, we denote by $\Fi(L)$ and $\Id(L)$ the collections of all filters and all ideals of $L$, respectively; and for a Boolean algebra $B$, we denote by $\Uf(B)$ the collection of all ultrafilters of $B$.

Let $L$ be a lattice. A \textit{completion} of $L$ is pair $\langle C,e\rangle$ such that $C$ is a complete lattice and $e\colon L\to C$ is a lattice embedding. A completion $\langle C,e\rangle$ of $L$ is said to be \textit{dense} if for every $c\in C$,
\begin{itemize}
	\item[] $c=\bigvee\left\{\bigwedge e[F]: F\in\Fi(L), \bigwedge e[F]\leq c\right\}$ and
	\item[] $c=\bigwedge\left\{\bigvee e[I]: I\in\Id(L), c\leq\bigvee e[I]\right\}$.
\end{itemize}
A completion $\langle C,e\rangle$ of $L$ is said to be \textit{compact} when for every $F\in\Fi(L)$ and every $I\in\Id(L)$, if $\bigwedge e[F]\leq\bigvee e[I]$, then $F\cap I\neq\emptyset$.

\begin{theorem}[\cite{GeHa01}]
For every lattice $L$, there exists a unique, up to isomorphism, dense and compact completion $\langle C,e\rangle$.
\end{theorem}

\begin{definition}[\cite{GeHa01}]
Let $L$ be a lattice. The unique dense and compact completion of $L$ is called the \textit{canonical extension} of $L$ and it is denoted by $L^{\sigma}$.
\end{definition}

The following results were obtained through the theory of topological representation for distributive lattices and Boolean algebras.

\begin{proposition}[\cite{GeJo94}]
Let $L$ be a bounded distributive lattice. Then, the canonical extension of $L$ is a completely distributive algebraic lattice.
\end{proposition}

\begin{proposition}[{\cite[pp. 908-910]{JoTa51}}]
Let $B$ be a Boolean algebra. Then, the canonical extension of $B$ is an atomic and complete Boolean algebra.
\end{proposition}

Now we are going to focus on the framework of Boolean algebras. 

Let $B_1$ and $B_2$ be Boolean algebras and let $\langle B_1^\sigma,\ph_1\rangle$ and $\langle B_2^\sigma,\ph_2\rangle$ be their canonical extensions, respectively. Let $h\colon B_1\to B_2$ be an order-preserving map. The \textit{canonical extension} of $h$ is the map $h^\sigma\colon B_1^{\sigma}\to B_2^\sigma$ defined by (see \cite{JoTa51})
\begin{equation}\label{equa:def of ext map}
h^\sigma(u)=\bigvee\left\{\bigwedge\{\ph_2(h(a)): a\in F\}: F\in\Fi(B_1), \bigwedge\ph_1[F]\leq u\right\} 
\end{equation}
for every $u\in B_1^\sigma$. The map $h^\sigma$ is order-preserving and extends the map $h$, that is, for every $a\in B_1$, $h^\sigma(\ph_1(a))=\ph_2(h(a))$. The notion of canonical extension for order-preserving maps was generalized to more general settings, for instance, for distributive lattices \cite{GeJo94,GeJo00,GeJo04}, for lattices \cite{GeHa01} and for partially ordered sets \cite{DuGePa05}.

It is well known \cite[Chapter 11]{DaPri02} that the category \textbf{BA} of Boolean algebras and Boolean homomorphisms is dually equivalent to the category \textbf{BS} of Stone spaces (also called Boolean spaces) and continuous maps. Let us describe the contravariant functors $(.)_*\colon\mathbf{BA}\to\mathbf{BS}$ and $(.)^*\colon\mathbf{BS}\to\mathbf{BA}$. Let $B$ be a Boolean algebra. Let $\ph\colon B\to\p\left(\Uf(B)\right)$ be the map define as follows: for $a\in B$,
\begin{equation}\label{equa:phi}
\ph(a)=\{u\in\Uf(B): a\in u\}.
\end{equation}
Then $B_*:=\langle\Uf(B),\tau_B\rangle$ where $\tau_B$ is the topology on $\Uf(B)$ generated by the base $\{\ph(a): a\in B\}$. If $h\colon B_1\to B_2$ is a Boolean homomorphism, then the dual $h_*\colon\Uf(B_2)\to\Uf(B_1)$ is defined as $h_*:=h^{-1}$. Let $X$ be a Stone space and let $\Clop(X)$ be the collection of all clopen subsets of $X$. Then $X^*:=\langle\Clop(X),\cap,\cup,{}^c,\emptyset,X\rangle$. If $f\colon X_1\to X_2$ is a continuous map, then $f^*\colon\Clop(X_2)\to\Clop(X_1)$ is defined as  $f^*:=f^{-1}$.

It was shown in \cite{JoTa51} that the canonical extension of a Boolean algebra $B$ is up to isomorphism $\langle \p(\Uf(B)),\ph\rangle$, where $\ph$ is defined by \eqref{equa:phi}. From now on, we will identify the canonical extension $B^\sigma$ of a Boolean algebra $B$ with $\langle\p(\Uf(B),\ph\rangle$\label{canonical ext}. Hence, by \eqref{equa:def of ext map}, the canonical extension of an order-preserving map $h\colon B_1\to B_2$ becomes
\begin{equation}\label{equa:def ext map 2}
h^\sigma(A)=\bigcup\left\{\bigcap\{\ph_2(h(a)): a\in F\}: F\in\Fi(B_1), \bigcap\ph_1[F]\subseteq A\right\}
\end{equation}
for every $A\in\p(\Uf(B_1))$.

Now, since $\p(\Uf(B))$ is in fact a Boolean algebra, we can consider its dual Stone space $\p(\Uf(B))_*$ and the map $\widehat{\ph}\colon\p(\Uf(B))\to\p(\p(\Uf(B))_*)$ defined as follows
\[
\widehat{\ph}(A)=\{\nabla\in\p(\Uf(B))_*: A\in\nabla\}.
\]

\section{The Stone-\v{C}ech compactification}\label{sec:compactification}

Our main references for the concepts and results considered in this section are \cite{En89} and \cite{GiMe60}. We assume that the reader is familiar with the theory of general topology.

Let $X$ and $Y$ be topological spaces. A map $f\colon X\to Y$  is said to be a \textit{homeomorphic embedding} if $f\colon X\to f[X]$ is a homeomorphism.

\begin{definition}
Let $X$ be a topological space. A \textit{compactification} of $X$ is a pair $\langle Y,c\rangle$ such that $Y$ is a compact Hausdorff  topological space and $c\colon X\to Y$ is a homeomorphic embedding where $c[X]$ is dense in $Y$.
\end{definition}

\begin{theorem}
A topological space $X$ has a compactification if and only if $X$ is a Tychonoff space.
\end{theorem}

Compactifications of a space $X$ will be denoted by $cX$, that is, $cX$ is a compact Hausdorff  space and $c\colon X\to cX$ is a homeomorphic embedding such that $\overline{c[X]}=cX$ (where $\overline{c[X]}$ is the topological closure of the set $c[X]$ in the space $cX$).

Let $X$ be a topological space. Let us denote by $\Com(X)$ the collection of all compactifications of $X$. Let $c_1X$ and $c_2X$ be two compactifications of $X$.  We say that $c_1X$ and $c_2X$ are \textit{equivalent} if there exists a homeomorphism $f\colon c_1X\to c_2X$ such that $f\circ c_1=c_2$. It is clear that the relation of ``being equivalent'' on $\Com(X)$ is an equivalence relation. Let us denote by $\com(X)$ the set of all equivalence classes. In the sequel we shall identify equivalent compactifications; any class of equivalent compactifications will be considered as a single compactification and denoted by $cX$, where $cX$ is an arbitrary compactification in this class.

Now we define the binary relation $\leq$ on $\com(X)$ as follows:
\begin{equation}
\begin{split}
c_2X\leq c_1X \quad \text{iff} \quad &\text{there exists a continuous map } f\colon c_1X\to c_2X\\
&\text{ such that } f\circ c_1=c_2.
\end{split}
\end{equation}
It is straightforward to show that $\leq$ is a partial order on $\com(X)$. 

\begin{theorem}
Let $X$ be a topological space. Every non-empty subfamily $\mathcal{C}_0$ of $\com(X)$ has a least upper bound with respect to the order $\leq$ in $\com(X)$.
\end{theorem}

\begin{corollary}
For every Tychonoff space $X$, there exists in $\com(X)$ the greatest element with respect to $\leq$.
\end{corollary}

\begin{definition}
Let $X$ be a Tychonoff space. The greatest element in $\com(X)$ is called the \textit{Stone-\v{C}ech compactification} of $X$ and it is denoted by $\beta X$.
\end{definition}

The Stone-\v{C}ech compactification has the following important universal mapping property.

\begin{theorem}\label{theo:unique extension}
Let $X$ be a Tychonoff space. If $f\colon X\to Y$ is a continuous map into a compact Hausdorff space $Y$, then there exists a unique continuous map $f^{\beta}\colon\beta X\to Y$ such that $f^{\beta}\circ\beta=f$, see Figure \ref{fig:beta extension}.
\end{theorem}

\begin{figure}
\centering
  \begin{tikzpicture}
\node (X) at (-1,1) {$X$};
\node (Y) at (1,1) {$Y$};
\node (bX) at (-1,-1) {$\beta X$};
\path[->,>=angle 90]
(X) edge node[above] {$f$} (Y);
\path[->,>=angle 90]
(X) edge node[left] {$\beta$} (bX);
\path[->,dashed,>=angle 90]
(bX) edge node[right] {$f^{\beta}$} (Y);
\end{tikzpicture}
\caption{The universal mapping property of the Stone-\v{C}ech compactification.}
\label{fig:beta extension}
\end{figure}
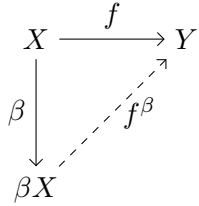

Now we move to consider the Stone-\v{C}ech compactification of discrete spaces. The following result is a useful characterization of the Stone-\v{C}ech compacification of a discrete space.

\begin{proposition}[\cite{En89}]\label{prop:SC-comp discrete spaces}
Let $X$ be a discrete topological space. Then, $\p(X)_*$ is (up to equivalent compactification) the Stone-\v{C}ech compactification of $X$ where $\beta\colon X\to\p(X)_*$ is defined by
\begin{equation}\label{equa:beta}
\beta(x)=\{A\in\p(X): x\in A\}.
\end{equation}
\end{proposition}

That is, the Stone-\v{C}ech compactification of a discrete space $X$ is the dual Stone space of the Boolean algebra $\p(X)$. From now on, we will identify the Stone-\v{C}ech compactification $\beta X$ of a discrete topological space $X$ with $\p(X)_*$. Thus, $\beta X=\p(X)_*$.

\section{A characterization of the canonical extensions of Boolean homomorphisms}\label{sec:characterization}

Let $B_1$ and $B_2$ be Boolean algebras and let $h\colon B_1\to B_2$ be a Boolean homomorphism. Then, by the duality between \textbf{BA} and \textbf{BS}, we have the dual $h_*\colon(B_2)_*\to(B_1)_*$ of $h$. Recall that $(B_i)_*=\Uf(B_i)$ for $i=1,2$. For every $i=1,2$, consider the Stone-\v{C}ech compactification $\beta\left(\Uf(B_i)\right)$ of the discrete space $\Uf(B_i)$. For $i=1,2$, we write $\beta_i\colon\Uf(B_i)\to\beta(\Uf(B_i))$ for the corresponding homeomorphic embeddings. Now we consider the composition $\beta_1\circ h_*\colon\Uf(B_2)\to\beta(\Uf(B_1))$. The map $\beta_1\circ h_*$ is trivially continuous. Then, by Theorem \ref{theo:unique extension}, there exists a unique continuous map $(\beta_1\circ h_*)^{\beta}\colon\beta(\Uf(B_2))\to\beta(\Uf(B_1))$ such that $(\beta_1\circ h_*)^{\beta}\circ\beta_2=\beta_1\circ h_*$, see Figure \ref{fig:beta ext composition}. We denote $h_*^{\beta}:=(\beta_1\circ h_*)^{\beta}$. 
\begin{figure}
\centering
  \begin{tikzpicture}
\node (X) at (-2,1) {$\Uf(B_2)$};
\node (Y) at (2,1) {$\beta(\Uf(B_1))$};
\node (V) at (-2,-1) {$\beta(\Uf(B_2))$};
\path[->,>=angle 90]
(X) edge node[above] {$\beta_1\circ h_*$} (Y);
\path[->,>=angle 90]
(X) edge node[left] {$\beta_2$} (V);
\path[->,dashed,>=angle 90]
(V) edge node[right] {$(\beta_1\circ h_*)^{\beta}$} (Y);
\end{tikzpicture}
\caption{}
\label{fig:beta ext composition}
\end{figure}
Thus 
\begin{equation}\label{equa:commutes}
h_*^\beta\circ\beta_2=\beta_1\circ h_*
\end{equation} 
and hence we obtain the commutative diagram in Figure \ref{fig:beta ext composition 2}.

\begin{figure}
\centering
  \begin{tikzpicture}
\node (X) at (-2,1) {$\Uf(B_2)$};
\node (Y) at (2,1) {$\Uf(B_1)$};
\node (V) at (-2,-1) {$\beta(\Uf(B_2))$};
\node (W) at (2,-1) {$\beta(\Uf(B_1))$};
\path[->,>=angle 90]
(X) edge node[above] {$h_*$} (Y);
\path[->,>=angle 90]
(X) edge node[left] {$\beta_2$} (V);
\path[->,dashed,>=angle 90]
(V) edge node[below] {$h_*^{\beta}$} (W);
\path[->,>=angle 90]
(Y) edge node[right] {$\beta_1$} (W);
\end{tikzpicture}
\caption{}
\label{fig:beta ext composition 2}
\end{figure}

Now recall that $\beta(\Uf(B_i))=\p(\Uf(B_i))_*$ for $i=1,2$, and thus we have
\begin{equation}
h_*^\beta\colon\p(\Uf(B_2))_*\to\p(\Uf(B_1))_*.
\end{equation} 
Moreover, recall that $\widehat{\ph_i}\colon\p(\Uf(B_i))\to\p\left(\beta(\Uf(B_i))\right)$ is given by $\widehat{\ph_i}(A)=\{\nabla\in\beta(\Uf(B_i)): A\in\nabla\}$, for $i=1,2$. Then, we can consider the Boolean dual of $h_*^{\beta}$:
\begin{equation}
(h_*^\beta)^*\colon\p(\Uf(B_1))\to\p(\Uf(B_2))
\end{equation}
where
\begin{equation}\label{equa:def f*b*}
(h_*^\beta)^*(A)=B \quad \text{if and only if} \quad \widehat{\ph_2}(B)=(h_*^\beta)^{-1}\left(\widehat{\ph_1}(A)\right).
\end{equation}

We summarize in the diagram of Figure \ref{fig:summarize} the previous constructions.

\begin{figure}
\centering
  \begin{tikzpicture}
\node (X1) at (-4.5,1.25) {$B_1$};
\node (X2) at (-1.5,1.25) {$\Uf(B_1)$};
\node (X3) at (1.5,1.25) {$\beta(\Uf(B_1))$};
\node (X4) at (4.5,1.25) {$\p(\Uf(B_1))$};
\node (Y1) at (-4.5,-1.25) {$B_2$};
\node (Y2) at (-1.5,-1.25) {$\Uf(B_2)$};
\node (Y3) at (1.5,-1.25) {$\beta(\Uf(B_2))$};
\node (Y4) at (4.5,-1.25) {$\p(\Uf(B_2))$};

\draw[->,>=angle 90]
(X1) edge[snake,decoration={amplitude=.4mm,segment length=2mm,post length=1mm}] node[above] {$(.)_*$} (X2);
\draw[->,>=angle 90]
(X2) edge node[above] {$\beta_1$} (X3);
\draw[->,>=angle 90]
(X3) edge[snake,decoration={amplitude=.4mm,segment length=2mm,post length=1mm}] node[above] {$(.)^*$} (X4);

\draw[->,>=angle 90]
(X1) edge node[left] {$h$} (Y1);
\draw[<-,>=angle 90]
(X2) edge node[left] {$h_*$} (Y2);
\draw[<-,>=angle 90]
(X3) edge node[right] {$h_*^\beta$} (Y3);
\draw[->,>=angle 90]
(X4) edge node[right] {$(h_*^\beta)^*$} (Y4);

\draw[->,>=angle 90]
(Y1) edge[snake,decoration={amplitude=.4mm,segment length=2mm,post length=1mm}] node[above] {$(.)_*$}  (Y2);
\draw[->,>=angle 90]
(Y2) edge node[below] {$\beta_2$} (Y3);
\draw[->,>=angle 90]
(Y3) edge[snake,decoration={amplitude=.4mm,segment length=2mm,post length=1mm}] node[above] {$(.)^*$} (Y4);
\end{tikzpicture}
\caption{Summary of the constructions made through of the Stone duality and the Stone-\v{C}ech compactification.}
\label{fig:summarize}
\end{figure}

\begin{remark}\label{rem:aux 1}
Let $A\in\p(\Uf(B_1))$ and $\Delta\in\beta(\Uf(B_2))$. Then,
\[
\Delta\in(h_*^\beta)^{-1}\left(\widehat{\ph_1}(A)\right) \iff h_*^\beta(\Delta)\in\widehat{\ph_1}(A) \iff A\in h_*^\beta(\Delta).
\]
\end{remark}

Recall that $B_i^\sigma=\p(\Uf(B_i))$ for $i=1,2$ and the canonical extension $h^\sigma\colon\p(\Uf(B_1))\to\p(\Uf(B_2))$ of $h$ is defined by  \eqref{equa:def ext map 2}.

Now, we are ready to establish and prove the main result of this paper.

\begin{theorem}\label{theo:principal}
Let $h\colon B_1\to B_2$ be a Boolean homomorphism between Boolean algebras. Then $h^\sigma=(h_*^\beta)^*$.
\end{theorem}

\begin{proof}
Let $A\in B_1^\sigma=\p(\Uf(B_1))$. Let $B\in B_2^\sigma=\p(\Uf(B_2))$ be such that
\begin{equation}\label{equa:aux 1}
(h_*^\beta)^*(A)=B\iff \widehat{\ph_2}(B)=(h_*^\beta)^{-1}\left(\widehat{\ph_1}(A)\right).
\end{equation}

Let $v\in\Uf(B_2)$. From \eqref{equa:commutes}, Remark \ref{rem:aux 1} and by \eqref{equa:aux 1}, we obtain the following equivalences:
\begin{multline}\label{equa:aux 2}
h_*(v)\in A \iff A\in\beta_1(h_*(v))=(\beta_1\circ h_*)(v) \iff A\in h_*^\beta(\beta_2(v))\\
 \iff \beta_2(v)\in(h_*^\beta)^{-1}\left(\widehat{\ph_1}(A)\right) \iff \beta_2(v)\in\widehat{\ph_2}(B)\\
 \iff v\in B \iff v\in(h_*^\beta)^*(A).
\end{multline}

Now, we assume that $v\in h^\sigma(A)$. Then, by \eqref{equa:def ext map 2}, there is $F\in\Fi(B_1)$ such that $\bigcap\ph_1[F]\subseteq A$ and $v\in\bigcap\{\ph_2(h(a)): a\in F\}$. Thus $v\in\ph_2(h(a))$ for all $a\in F$, and this implies that $F\subseteq h^{-1}[v]$. As $h$ is a Boolean homomorphism, we have that $h_*(v)=h^{-1}[v]\in\Uf(B_1)$. Then, since $\bigcap\ph_1[F]\subseteq A$, it follows that $h_*(v)\in A$. Thus, by \eqref{equa:aux 2}, we obtain $v\in(h_*^\beta)^*(A)$. Hence, we have proved that $h^\sigma(A)\subseteq(h_*^\beta)^*(A)$.

In order to prove the inverse inclusion, assume that $v\in(h_*^\beta)^*(A)$. By \eqref{equa:aux 2}, we have $h_*(v)\in A$. Let $F:=h_*(v)=h^{-1}[v]\in\Uf(B_1)$ and so, in particular, $F\in\Fi(B_1)$. Since $F$ is maximal, it follows that $\bigcap\ph_1[F]=\{F\}$. Then $\bigcap\ph_1[F]\subseteq A$. Let $a\in F=h^{-1}[v]$. Thus $v\in\ph_2(h(a))$. Hence $v\in\bigcap\{\ph_2(h(a)): a\in F\}$. Then, we have proved that there exists $F\in\Fi(B_1)$ such that $\bigcap\ph_1[F]\subseteq A$ and $v\in\bigcap\{\ph_2(h(a)): a\in A$. Hence, by \eqref{equa:def ext map 2}, we have $v\in h^\sigma(A)$. We thus obtain $(h_*^\beta)^*(A)\subseteq h^\sigma(A)$. Therefore, $(h_*^\beta)^*(A) = h^\sigma(A)$.
\end{proof}

Let $X$ and $Y$ be discrete topological spaces and let $f\colon X\to Y$ be a function. Then, by Theorem \ref{theo:unique extension}, there exists a continuous function $f^{\beta}\colon\beta X\to\beta Y$ such that the following diagram

\begin{figure}[h]
\centering
  \begin{tikzpicture}
\node (X) at (-1,1) {$X$};
\node (Y) at (1,1) {$Y$};
\node (bX) at (-1,-1) {$\beta X$};
\node (bY) at (1,-1) {$\beta Y$};
\path[->,>=angle 90]
(X) edge node[above] {$f$} (Y);
\path[->,>=angle 90]
(X) edge node[left] {$\beta_X$} (bX);
\path[->,>=angle 90]
(Y) edge node[right] {$\beta_Y$} (bY);
\path[->,dashed,>=angle 90]
(bX) edge node[below] {$f^{\beta}$} (bY);
\end{tikzpicture}
\caption{A function between discrete spaces and their corresponding Stone-\v{C}ech-compactifications}
\label{fig:beta extension 2}
\end{figure}
\noindent commutes. Recall (Proposition \ref{prop:SC-comp discrete spaces}) that $\beta X=\p(X)_*=\Uf\left(\p(X)\right)$.

\begin{lemma}[\cite{HiStr12}]\label{lem:aux 1}
In the above hypotheses, we have:
\begin{enumerate}[(1)]
	\item $f^{\beta}(\nabla)=\{B\subseteq Y: f^{-1}[B]\in\nabla\}$, for all $\nabla\in\beta X$.
	\item If $\nabla\in\beta X$ and $A\in\nabla$, then $f[A]\in f^{\beta}(\nabla)$.
\end{enumerate}
\end{lemma}

\begin{proposition}
Let $X$ and $Y$ be discrete topological spaces and let $f\colon X\to Y$ be a function.
\begin{enumerate}[(1)]
	\item If $f$ is a one-to-one function, then $f^{\beta}$ is a one-to-one function.
	\item If $f$ is onto, then $^{\beta}$ is onto.
	\item If $f$ is a one-to-one correspondence, then $f^{\beta}$ is an homeomorphism.
\end{enumerate}
\end{proposition}

\begin{proof}
(1) Assume that $f$ is a one-to-one function. Let $\nabla_1,\nabla_2\in\beta X$ be such that $f^{\beta}(\nabla_1)=f^{\beta}(\nabla_2)$. Let $A\in\nabla_1$. By Lemma \ref{lem:aux 1}, we have $f[A]\in f^{\beta}(\nabla_1)$. Thus $f[A]\in f^{\beta}(\nabla_2)$. Then, by Lemma \ref{lem:aux 1} again and since $f$ is a one-to-one function, it follows that $A=f^{-1}\left[f[A]\right]\in\nabla_2$. We have proved that $\nabla_1\subseteq\nabla_2$. Hence, since $\nabla_1$ is an ultrafilter of $\p(X)$, we obtain that $\nabla_1=\nabla_2$. Therefore, $f$ is one-to-one.

(2) Assume that $f$ is onto. Let $\Delta\in\beta Y$. Now let 
\[
\F:=\{A\subseteq X: f^{-1}[B]\subseteq A \ \text{ for some } \ B\in\Delta\}.
\]
It follows straightforward that $\F$ is a filter of the Boolean algebra $\p(X)$. Since $\Delta$ is an ultrafilter of the Boolean algebra $\p(Y)$ and $f$ is an onto function, it follows that $\F$ is a proper filter of $\p(X)$. Let $\nabla$ be an ultrafilter of $\p(X)$ such that $\F\subseteq\nabla$. Since $f$ is onto, it follows that the set $\{f[A]: A\in\nabla\}$ is a proper filter of $\p(Y)$ and $\Delta\subseteq\{f[A]: A\in\nabla\}$. Hence, because $\Delta$ is an ultrafilter of $\p(Y)$, we obtain that $\Delta=\{f[A]: A\in\nabla\}$. Now we are ready to show that $f^{\beta}(\nabla)=\Delta$. Let $B\in f^{\beta}(\nabla)$. By Lemma \ref{lem:aux 1}, we have $f^{-1}[B]\in\nabla$. So $f[f^{-1}[B]]\in\Delta$. Since $f$ is onto, it follows that $B\in\Delta$. Now, let $B\in\Delta$. Thus $B=f[A]$ for some $A\in\nabla$. As $A\subseteq f^{-1}\left[f[A]\right]=f^{-1}[B]$, we have $f^{-1}[B]\in\nabla$. Then, by Lemma \ref{lem:aux 1}, we obtain $B\in f^{\beta}(\nabla)$. Hence, $f^{\beta}(\nabla)=\Delta$. Therefore, $f^{\beta}$ is onto.

(3) Lastly, assume that $f$ is a one-to-one correspondence. By (1) and (2), we have that $f^{\beta}$ is a one-to-one correspondence. Moreover, we know that $f^{\beta}$ is a continuous map. Then, since the spaces $\beta X$ and $\beta Y$ are Hausdorff and compact, we obtain that $f^{\beta}$ is a homeomorphism.
\end{proof}

Now from the previous proposition, Theorem \ref{theo:principal} and using the diagram in Figure \ref{fig:summarize}, it follows the following corollary.

\begin{corollary}
Let $B_1$ and $B_2$ be Boolean algebras and let $h\colon B_1\to B_2$ be a homomorphism.
\begin{enumerate}[(1)]
	\item If $h$ is a one-to-one homomorphism, then $h^{\sigma}$ is a one-to-one homomorphism.
	\item If $h$ is an onto homomorphism, then $h^{\sigma}$ is an onto homomorphism.
	\item If $h$ is an isomorphism, then $h^{\sigma}$ is an isomorphism.
\end{enumerate}
\end{corollary}


\end{document}